\documentclass[11pt]{amsart}

\usepackage{latexsym}
\usepackage{amssymb}
\usepackage{amsmath}

\newtheorem{theorem}{Theorem}[section]
\newtheorem{lemma}[theorem]{Lemma}
\newtheorem{proposition}[theorem]{Proposition}
\newtheorem{corollary}[theorem]{Corollary}

\theoremstyle{definition}
\newtheorem{definition}[theorem]{Definition}

\theoremstyle{remark}

\def\N{\mathbb{N}}
\def\Z{\mathbb{Z}}

\def\hN{{{}^*\N}}

\def\hhN{{}^{**}\N}

\def\hA{{}^*A}

\def\hX{{}^*X}
\def\hf{{}^*f}

\def\F{\mathcal{F}}

\def\U{\mathcal{U}}
\def\V{\mathcal{V}}
\def\W{\mathcal{W}}

\def\UU{\mathfrak{U}}

\def\ueq{{\,{\sim}_{{}_{\!\!\!\!\! u}}\;}}
\def\nueq{{\,{\not\sim}_{{}_{\!\!\!\! u}}\,}}

\def\ueqs{{\,{\approx}_{{}_{\!\!\!\!\! u}}\;}}

\begin{document}

\title[Iterated hyper-extensions]
{Iterated hyper-extensions and an idempotent ultrafilter proof
of Rado's theorem}

\author{Mauro Di Nasso}
\address{Dipartimento di Matematica,
Universit\`{a} di Pisa, Italy.} \email{dinasso@dm.unipi.it}




\maketitle

\begin{abstract}
By using nonstandard analysis, and in particular iterated
hyper-extensions, we give foundations to a
peculiar way of manipulating ultrafilters on the natural numbers
and their pseudo-sums. The resulting formalism is
suitable for applications in Ramsey theory of numbers.
To illustrate the use of our technique,
we give a (rather) short proof of Milliken-Taylor's Theorem,
and a ultrafilter version of Rado's theorem about partition
regularity of diophantine equations.
\end{abstract}

\bigskip
\section*{Introduction}

The algebraic structure
on the space of ultrafilters
$\beta\N$ as given by the pseudo-sum operation
$\U\oplus\V$, and the related generalizations,
have been deeply investigated during the
last thirty years, revealing a powerful tool for applications
in Ramsey theory and combinatorial number theory
(see the monograph \cite{hs}).
The aim of this paper is to introduce a peculiar
formalism grounded on the use of the
hyper-natural numbers of nonstandard analysis,
that allows to manipulate ultrafilters on $\N$
and their pseudo-sums in a simplified manner.
Especially, we shall be interested in
linear combinations $a_0\U\oplus\ldots\oplus a_k\U$
of a given idempotent ultrafilter $\U$.
To illustrate the use of our technique,
we shall give a nonstandard proof of Ramsey Theorem,
and a (rather) short proof of Milliken-Taylor's
Theorem, a strengthening of the celebrated Hindman's Theorem.
Moreover, we shall also prove the following
ultrafilter version of Rado's Theorem, that seems
to be new.

\medskip
\noindent
\textbf{Theorem.}
{\emph{
Let $c_1 X_1+\ldots + c_k X_k=0$ be a diophantine
equation with $c_1+\ldots+c_k=0$ and $k>2$.
Then there exists $a_0,\ldots,a_{k-2}\in\N$
such that for every idempotent ultrafilter $\U$,
the corresponding linear combination
$$\V=a_0\U\oplus \ldots\oplus a_{k-2}\U$$
witnesses the injective partition regularity of the
given equation, \emph{i.e.} for every $A\in\V$ there exist
distinct elements $x_1,\ldots,x_k\in A$
with $c_1 x_1+\ldots+c_k x_k=0$.}

\medskip
At the end of the paper, some hints are given
for further possible applications and developments
of the introduced nonstandard technique.

\smallskip
We assume the reader to be familiar with the notion
of \emph{ultrafilter},
and with the basics of \emph{nonstandard analysis}.
In particular, we shall call \emph{star map} or
\emph{nonstandard embedding} a function $A\mapsto\hA$
that associates to each mathematical
object $A$ under consideration its \emph{hyper-extension}
$\hA$, and that satisfies the \emph{transfer principle}.
Excellent references for the foundations
of nonstandard analysis are \cite{ck} \S 4.4,
where the classical superstructure approach is presented,
and the textbook \cite{go}, grounded on the ultrapower
construction.
The peculiarity of our nonstandard approach is
that we shall use iterated hyper-extensions
(see the discussion in Section \ref{sec-iterated}).



\bigskip
\section{$u$-equivalent polynomials}\label{sec-uequivalence}

Before starting to work in a nonstandard setting,
in this section we present a general result about
linear combinations of a given idempotent ultrafilter
(see below for the definition),
whose proof will be given in Section \ref{rado}.
As a consequence of its, we prove an
ultrafilter version of Rado's theorem.

\smallskip
Throughout the paper, $\N$ will denote
the set of \emph{positive} integers,
and $\N_0=\N\cup\{0\}$ the set of \emph{non-negative} integers.

\smallskip
Recall the \emph{pseudo-sum} operation between ultrafilters on $\N_0$:
$$A\in\U\oplus\V\ \Longleftrightarrow\ \{n\in\N_0\mid A-n\in\V\}\in\U,$$
where $A-n=\{m\in\N_0\mid m+n\in A\}$ is the
left-ward \emph{shift} of $A$ by $n$. It can be readily verified
that $\U\oplus\V$ is actually an ultrafilter,
and that the pseudo-sum operation is \emph{associative}.
If one identifies every \emph{principal ultrafilter}
$\UU_n=\{A\subseteq\N_0\mid n\in A\}$ with its
generator $n\in\N_0$, then it is
readily seen that the pseudo-sum extends the usual
addition, \emph{i.e.} $\UU_n\oplus\UU_m=\UU_{n+m}$;
moreover, $\U\oplus\UU_0=\UU_0\oplus\U=\U$ for all $\U$.
(In fact, it can be proved that the center of $(\beta\N_0,\oplus)$
is the family $\{\UU_n\mid n\in\N_0\}$ of the principal
ultrafilters. A nonstandard proof of this fact can be found in
\cite{dn5}.)

\smallskip
Given an ultrafilter $\U$ on $\N$ and a natural number $h$,
the product $h\,\U$ is the ultrafilter defined by putting
$$A\in h\,\U\ \Longleftrightarrow\ A/h=\{n\mid n\,h\in A\}\in\U.$$
Notice that $0\,\U=\UU_0$ and $1\,\U=\U$ for every $\U$.

\smallskip
Particularly relevant for applications are the \emph{idempotent ultrafilters},
namely the \emph{non-principal} ultrafilters $\U$ such that $\U\oplus\U=\U$.
We remark that their existence is a non-trivial result whose proof
requires repeated applications of Zorn's lemma.
(To be precise, also the principal ultrafilter $\UU_0$ has
the property $\UU_0=\UU_0\oplus\UU_0$, but it is not
usually considered as an ``idempotent" in the literature.)

\smallskip
Let us now introduce an equivalence relation on
the \emph{strings} (\emph{i.e.}, finite sequences) of integers.

\medskip
\begin{definition}\label{ueqs}
The $u$-\emph{equivalence} $\ueqs$ between
strings of integers is the smallest
equivalence relation such that:

\smallskip
\begin{itemize}
\item
The empty string $\varepsilon\ueqs\langle 0\rangle$.

\smallskip
\item
$\langle a\rangle\ueqs\langle a,a\rangle$ for all $a\in\Z$.

\smallskip
\item
$\ueqs$ is coherent with \emph{concatenations}\footnote
{~Recall that if $\sigma=\langle a_0,\ldots,a_k\rangle$ and
$\tau=\langle b_0,\ldots,b_h\rangle$,
then their concatenation is defined as
$\sigma^{\frown}\tau=\langle a_0,\ldots,a_k,b_0,\ldots,b_h\rangle$.},
\emph{i.e.}
$$\sigma\ueqs\sigma'\ \text{and }\tau\ueqs\tau'\ \Longrightarrow\
\sigma^{\frown}\tau\,\ueqs\,\sigma'^\frown\tau'.$$
\end{itemize}

\smallskip
Two polynomials $P(X)=\sum_{i=0}^n a_i X^i$ and
$Q(X)=\sum_{j=0}^m b_j X^j$ in $\Z[X]$ are $u$-\emph{equivalent}
when the corresponding strings of coefficients are $u$-equivalent:
$$\langle a_0,\ldots,a_n\rangle\,\ueqs\,
\langle b_0,\ldots,b_m\rangle$$
\end{definition}

\medskip
So, $\ueqs$-equivalence between strings
is preserved by inserting or removing zeros,
by repeating finitely many times a term or,
conversely, by shortening a block of consecutive equal
terms. \emph{E.g.} $\langle 3,0,0,-4,1,1\rangle\,\ueqs\,
\langle 0,3,-4,-4,1\rangle$ and
$\langle 2,2,0,0,7,7,3\rangle\,\ueqs\,
\langle 2,7,3\rangle$, and hence:

\smallskip
\begin{itemize}
\item
$X^5+X^4-4X^3+3\ \ueqs\, X^4-4X^3-4X^2+3X$

\smallskip
\item
$3X^6+7X^5+7X^4+2X+2\ \ueqs\, 3X^2+7X+2$, \emph{etc}.
\end{itemize}

\medskip
As an application of nonstandard analysis, in
Section \ref{rado} the following will be proved:

\medskip
\noindent
\textbf{Theorem \ref{linear-pr}}
\emph{Let $a_0,\ldots,a_n\in\N_0$, and assume that
there exist [distinct] polynomials $P_i(X)$ such that
$$P_1(X)\,\ueqs\, \ldots\,\ueqs\, P_k(X)\,\ueqs\, \sum_{i=0}^n a_iX^i\ \ \text{and}\ \
c_1 P_1(X)+\ldots+c_k P_k(X)=0.$$
Then for every idempotent ultrafilter $\U$ and for every
$A\in a_0\U\oplus\ldots\oplus a_n\U$,
there exist [distinct]
$x_i\in A$ such that $c_1 x_1+\ldots+c_k x_k=0$.}

\medskip
We derive here a straight consequence of the above theorem
which is a ultrafilter version of Rado's theorem.

\medskip
\begin{theorem}\label{th-rado}
Let $c_1 X_1+\ldots + c_k X_k=0$ be a diophantine
equation with $c_1+\ldots+c_k=0$ and $k>2$.
Then there exists $a_0,\ldots,a_{k-2}\in\N$
such that for every idempotent ultrafilter $\U$,
the corresponding linear combination
$$\V=a_0\U\oplus \ldots\oplus a_{k-2}\U$$
witnesses the injective partition regularity of the
given equation, \emph{i.e.}, for every $A\in\V$ there exist distinct
$x_i\in A$ such that $c_1 x_1+\ldots+c_k x_k=0$.
\end{theorem}

\begin{proof}
For arbitrary $a_0,\ldots,a_{k-2}$, consider the following polynomials:
$$\!\!\!\!\!\!
\tiny{
\begin{array}{lllllllllllllll}
P_1(X) &=& a_0 &+& a_1 X &+& a_2 X^2 &+& \ldots &+& a_{k-3} X^{k-3} &+& a_{k-2}X^{k-2} &+& a_{k-2}X^{k-1} \\
P_2(X) &=& a_0 &+& a_1 X &+& a_2 X^2 &+& \ldots &+& a_{k-3}X^{k-3} &+& 0 &+& a_{k-2}X^{k-1} \\
P_3(X) &=& a_0 &+& a_1 X &+& a_2 X^2 &+& \ldots &+& 0 &+& a_{k-3}X^{k-2} &+& a_{k-2}X^{k-1} \\
\vdots &{}& \vdots &{}& \vdots &{}& \vdots &{}& \vdots &{}& \vdots &{}& \vdots &{}& \vdots \\
P_{k-2}(X) &=& a_0 &+& a_1 X &+& 0 &+& a_2 X^3 &+& \ldots &+& a_{k-3}X^{k-2} &+& a_{k-2}X^{k-1} \\
P_{k-1}(X) &=& a_0 &+& 0 &+& a_1 X^2 &+& a_2 X^3 &+& \ldots &+& a_{k-3}X^{k-2} &+& a_{k-2}X^{k-1} \\
P_k(X) &=& a_0 &+& a_0 X &+& a_1 X^2 &+& a_2 X^3 &+& \ldots &+& a_{k-3}X^{k-2} &+& a_{k-2} X^{k-1}
\end{array}}
$$

Notice that $P_1(X)\,\ueqs\,\ldots\,\ueqs\,P_k(X)\,\ueqs\,\sum_{i=0}^{k-2}a_i X^i$.
In order to apply Theorem \ref{linear-pr}, we need to find
suitable coefficients $a_0,\ldots,a_{k-2}$
in such a way that the linear combination
$c_1 P_1(X)+\ldots+c_k P_k(X)=0$.
It is readily verified that this happens
if and only if the following conditions are fulfilled:

$$\begin{cases}
(c_1+\ldots+c_k)\cdot a_0 = 0 \\
(c_1+\ldots+c_{k-2})\cdot a_1 + c_k\cdot a_0 = 0 \\
(c_1+\ldots+c_{k-3})\cdot a_2+ (c_{k-1}+c_k)\cdot a_1 = 0 \\
{}\quad\quad\quad \vdots \\
c_1\cdot a_{k-2}+(c_3+\ldots+c_k)\cdot a_{k-3} = 0 \\
(c_1+\ldots+c_k)\cdot a_{k-2} = 0
\end{cases}$$

\smallskip
The first and the last equations are trivially satisfied because of the
hypothesis $c_1+\ldots+c_k=0$. Now assume without loss
of generality that the coefficients
$c_1\ge \ldots\ge c_k$ are arranged in non-increasing order.
It can be verified in a straightforward manner that
the remaining $k-2$ equations are satisfied by (infinitely many)
suitable $a_0,\ldots,a_{k-2}\in\N$, \emph{e.g.}
$$(*)\quad
\begin{cases}
a_0\ =\ c_1\cdot(c_1+c_2)\cdot\ldots\cdot(c_1+\ldots+c_{k-2}) \\
a_i\ =\ b_i\cdot b'_i\ \text{ for } 0<i<k-2\ \ \text{where }\\
\quad\quad b_i=c_1\cdot(c_1+c_2)\cdot\ldots\cdot(c_1+\ldots+c_{k-2-i})\ \text{and} \\
\quad\quad b'_i=(-1)^i\cdot c_k\cdot(c_k+c_{k-1})\cdot\ldots\cdot(c_k+\ldots+c_{k+1-i}) \\
a_{k-2}\ =\ (-1)^{k-2}\cdot c_k\cdot(c_k+c_{k-1})\cdot\ldots\cdot(c_k+\ldots+c_3)
\end{cases}$$

Remark that by the hypothesis $c_1+\ldots+c_k=0$ together
with the assumption $c_1\ge\ldots\ge c_k$, it follows
that all $a_i>0$.
Finally, remark also that since all coefficients $a_0,\ldots,a_{k-2}\ne 0$,
the polynomials $P_i(X)$ are mutually distinct.
\end{proof}

\medskip
For instance, consider the diophantine equation
$$3X_1+X_2+X_3-X_4-4X_5\ =\ 0$$
where coefficients are arranged in non-increasing order and their
sum equals zero.
By using $(*)$ in the above proof, we obtain that
$a_0=60$, $a_1=48$, $a_2=60$, $a_3=80$.
So, for every idempotent ultrafilter $\U$
and for every set
$$A\in 60\,\U\oplus 48\,\U\oplus 60\,\U\oplus 80\,\U$$
we know that there exist distinct
$x_i\in A$
such that
$3x_1+x_2+x_3-x_4-4 x_5=0$.

\bigskip
\section{Iterated hyper-extensions}\label{sec-iterated}

For our purposes, we shall need to work in models of nonstandard analysis
where hyper-extensions can be iterated, so that
one can consider, \emph{e.g.}, the set of hyper-hyper-natural numbers ${}^{**}\N$,
the hyper-hyper-hyper-natural numbers ${}^{***}\N$, the
hyper-extension ${}^*\nu$ of an hyper-natural number $\nu$,
and so forth.
Moreover, we shall use the $\mathfrak{c}^+$-\emph{enlargement property},
namely the property that intersections
$\bigcap_{F\in\F}{}^*F$ are nonempty for all
families $|\F|\le\mathfrak{c}$ of cardinality at most the
continuum, and which satisfy the \emph{finite intersection property}
(\emph{i.e.} $A_1\cap\ldots\cap A_n\ne\emptyset$ for every choice
of finitely many $A_i\in\F$).

\smallskip
In full generality, all the above requirements are fulfilled
by taking a $\mathfrak{c}^+$-enlarging nonstandard embedding
$$*:\mathbf{V}\,\longrightarrow\,\mathbf{V}$$
which is defined on the \emph{universal class} $\mathbf{V}$ of all sets.
It is now a well-known fact in nonstandard set theory that
such ``universal" nonstandard embeddings can be constructed
within conservative extensions of ZFC where the regularity axiom
is replaced by a suitable \emph{anti-foundation axiom};
see \emph{e.g.} \cite{bhr}.\footnote{
~In summary, one takes \emph{transitive Mostowski collapses} of
\emph{ultrapowers} $\mathbf{V}^I/\U$ of the universe.
For any given cardinal $\kappa$, the $\kappa$-enlargement property
is obtained by picking
a $\kappa$-\emph{regular} ultrafilter $\U$.}
(For a comprehensive treatment of nonstandard set theories,
we refer the interested reader to the monography \cite{kr}.)

\smallskip
A suitable axiomatic framework is the nonstandard
set theory $\text{ZFC}[\Omega]$ of \cite{dn4},
which includes all axioms of Zermelo-Fraenkel theory ZFC
with choice with the only exception
of the regularity axiom, and where for every ``$\in$-definable"
cardinal $\kappa$, one has a nonstandard embedding
$J_\kappa:\mathbf{V}\to\mathbf{V}$
of the universe into itself that satisfies the
$\kappa$-enlargement property.\footnote
{~Indeed, the separation and replacement
schemas hold for all $\in$-$*$-formulas,
and $J_\kappa$ is postulated to satisfy
$\kappa$-\emph{saturation},
a stronger property than $\kappa$-enlargement.}
The resulting theory is conservative over ZFC.
(See also the related axiomatics ${}^*\text{ZFC}$ \cite{dn1,dn2}
and Alpha-Theory \cite{dn3}).

\smallskip
Another suitable setting where iterated hyper-extensions
can be considered was introduced by V. Benci in \cite{be}:
it consists in a special version of the \emph{superstructure approach}
$*:V(X)\to V(X)$ where the standard universe and the
nonstandard universe coincide. The limitation here is that
superstructures $V(X)$ only satisfies a fragment of ZFC
(\emph{e.g.}, replacement fails and there are no infinite
ordinals in $V(X)$).

\smallskip
We stress that working with iterated hyper-extensions requires
caution. To begin with, recall that in nonstandard analysis one has
that ${}^*n=n$ for all natural numbers $n\in\N$;
however, the same property cannot be extended to the
hyper-naturals numbers $\xi\in\hN$. Indeed, by \emph{transfer}
one can easily show that ${}^*\xi>\xi$ for all infinite
$\xi\in\hN$; more generally, the following facts hold:

\smallskip
\begin{itemize}
\item
$\hN\varsubsetneq\hhN$.

\smallskip
\item
If $\xi\in\hN\setminus\N$ then ${}^*\xi\in\hhN\setminus\hN$.

\smallskip
\item
$\hN$ is an initial segment of $\hhN$, \emph{i.e.}
$\xi<\nu$ for every $\xi\in\hN$ and for every
$\nu\in\hhN\setminus\hN$.
\end{itemize}

\medskip
Let $f:\N_0\to\N_0$. By transferring the fact that $\hf$ and $f$
agree on $\N_0$, one gets that

\smallskip
\begin{itemize}
\item
$({}^*\hf)(\xi)=\hf(\xi)$ for all $\xi\in\hN_0$.
\end{itemize}

\smallskip
Remark that ${}^*[\hf(\xi)]=({}^*\hf)({}^*\xi)$, but
in general $({}^*\hf)({}^*\xi)\ne ({}^*\hf)(\xi)$.

\smallskip
Now
denote by ``$k*$" the $k$-times iterated star map, \emph{i.e.}
$$\begin{cases}
{}^{0*}A\,=\,A
\\
{}^{(k+1)*}A\,=\,{}^*\!\left({}^{k*}A\right).
\end{cases}$$

\smallskip
As a first application of iterated hyper-extensions,
we now present a nonstandard proof of Ramsey theorem.

\medskip
\begin{lemma}\label{lemma1}
Let $A\subseteq\N^k$.
If there exists an infinite $\xi\in\hN$
such that $(\xi,{}^*\xi,\ldots,{}^{(k-1)*}\xi)\in{}^{k*}A$
then there exists an infinite set of natural numbers
$$H\ =\ \{h_1<h_2<\ldots<h_n<h_{n+1}<\ldots\ \}$$
such that $(h_{n_1},h_{n_2},\ldots,h_{n_k})\in A$
for all $n_1<n_2<\ldots<n_k$.
\end{lemma}

\begin{proof}
To simplify notation, let us only consider here the particular
case $k=3$. The general result is proved exactly in
the same fashion, only by using a heavier notation.
So, let us assume that $(\xi,{}^*\xi,{}^{**}\xi)\in{}^{***}A$.
For $n,n'\in\N$, let:

\smallskip
\begin{itemize}
\item
$X=\{n\in\N\mid (n,\xi,{}^*\xi)\in{}^{**}A\}$\,;

\smallskip
\item
$X_n=\{n'\in\N\mid(n,n',\xi)\in\hA\}$\,;

\smallskip
\item
$X_{nn'}=\{n''\in\N\mid (n,n',n'')\in A\}$.
\end{itemize}

\smallskip
The corresponding hyper-extensions are described as follows:

\smallskip
\begin{itemize}
\item
$\hX=\{\eta\in\hN\mid (\eta,{}^*\xi,{}^{**}\xi)\in{}^{***}A\}$\,;

\smallskip
\item
$\hX_n=\{\eta\in\hN\mid(n,\eta,{}^*\xi)\in{}^{**}A\}$\,;

\smallskip
\item
$\hX_{n,n'}=\{\eta\in\hN\mid(n,n',\eta)\in\hA\}$.
\end{itemize}

\smallskip
Notice that
$n\in X\Leftrightarrow\xi\in\hX_n$, and
$n'\in X_n\Leftrightarrow\xi\in\hX_{nn'}$.
By the hypothesis, we have that $\xi\in\hX$, so $X$ is an infinite set
and we can pick an element $h_1\in X$.
Now, $\xi\in\hX\cap\hX_{h_1}$ implies that
$X\cap X_{h_1}$ is infinite, and so we can pick
an element $h_2>h_1$ in that intersection.
But then $\xi\in\hX\cap\hX_{h_1}\cap\hX_{h_2}\cap\hX_{h_1h_2}$,
and so we can pick an element $h_3>h_2$ in the intersection
$X\cap X_{h_1}\cap X_{h_2}\cap X_{h_1h_2}$.
In particular, $(h_1,h_2,h_3)\in A$.
An increasing sequence $\langle h_n\mid n\in\N\rangle$
that satisfies the desired property is obtained by iterating
this procedure, where at each step $n$ one has
$$\xi\in\hX\,\cap\,\bigcap_{1\le i\le n}\!\!\hX_{h_i}\,\,\cap\,
\bigcap_{1\le i<j\le n}\!\!\hX_{h_ih_j},$$
and $h_{n+1}>h_n$ is picked in the infinite intersection
$$h_{n+1}\in X\,\cap\,\bigcap_{1\le i\le n}\!\!X_{h_i}\,\,\cap\,
\bigcap_{1\le i<j\le n}\!\!X_{h_ih_j}.$$
\end{proof}

\medskip
As a straight corollary, one obtains:

\medskip
\begin{theorem}[Ramsey]
Let $[\N]^k=C_1\sqcup\ldots\sqcup C_r$
be a finite partition of the $k$-sets of natural numbers.\footnote
{~A $k$-\emph{set} is a set with exactly
$k$-many elements.}
Then there exists an infinite
$H\subseteq\N$ such that all its $k$-sets are monochromatic,
\emph{i.e.} $[H]^k\subseteq C_i$ for some $i$.
\end{theorem}

\begin{proof}
Identify the family of $k$-sets $[\N]^k$ with the upper-diagonal
in the Cartesian product $\N^k$:
$$\left\{(n_1,n_2,\ldots,n_k)\in\N^k\,\big|\,
n_1<n_2<\ldots<n_k\right\}.$$

By applying \emph{transfer} to the $k$-iterated star map,
one gets that ${}^{k*}([\hN]^k)=[{}^{k*}\N]^k$,
and one has the following finite coloring:
$$[{}^{k*}\N]^k\ =\ {}^{k*}C_1\sqcup\ldots\sqcup{}^{k*}C_r.$$

Now fix any infinite $\xi\in\hN$, and let
$i$ be such that the ordered $k$-tuple
$(\xi,{}^*\xi,\ldots,{}^{(k-1)*}\xi)\in{}^{k*}C_i$
(notice that $\xi<{}^*\xi<\ldots<{}^{(k-1)*}\xi$).
By the Lemma \ref{lemma1}, we get the existence of an infinite
$H\subseteq\N$ such that $[H]^k\subseteq[C_i]^k$.
\end{proof}

\bigskip
\section{
Hyper-natural numbers as representatives of ultrafilters}

Recall that
there is a canonical way of associating
an ultrafilter to every element $\alpha\in\hN_0$
(see \emph{e.g.} \cite{ch,nr,dn5}).
Namely, one takes the family of those sets of natural
numbers whose hyper-extensions contain $\alpha$:
$$\UU_\alpha\ =\ \{A\subseteq\N_0\mid \alpha\in\hA\}.$$

It is readily verified from the properties of
hyper-extensions that $\UU_\alpha$ is indeed an
ultrafilter, called the \emph{ultrafilter generated}
by $\alpha$. Notice that $\UU_\alpha$ is principal if
and only if $\alpha\in\N_0$ is finite.
Notice also that if $\U=\UU_\alpha$ is generated by $\alpha$,
then $h\,\U$ is generated by $h\alpha$.

\medskip
\begin{definition}
We say that two elements $\alpha,\beta\in\hN_0$ are $u$-\emph{equivalent},
and write $\alpha\ueq\beta$, if they generate the same ultrafilter:
$\UU_\alpha=\UU_\beta$.
\end{definition}

\medskip
So, $\alpha\ueq\beta$ when $\alpha\in\hA\Leftrightarrow\beta\in\hA$
for all $A\subseteq\N_0$.
Since every ultrafilter $\U$ on $\N_0$ is a family
of $\mathfrak{c}$-many sets with the
finite intersection property, by the hypothesis of
$\mathfrak{c}^+$-\emph{enlargement}
there exists an element $\alpha\in\bigcap_{A\in\U}\hA$.
This means that every ultrafilter $\U=\UU_\alpha$ is generated by some
number $\alpha\in\hN_0$.

\smallskip
The following properties are readily verified:

\medskip
\begin{proposition}
Let $\alpha\ueq\alpha'$ be two $u$-equivalent hyper-natural
numbers, and let $n\in\N$ be finite. Then:

\smallskip
\begin{enumerate}
\item
$\alpha\pm n\,\ueq\,\alpha'\pm n$,
\item
$n\cdot\alpha\,\ueq\, n\cdot\alpha'$,
\item
$\alpha/n\,\ueq\,\alpha'/n$, provided $\alpha$ is divisible by $n$.
\end{enumerate}
\end{proposition}

\medskip
Remark that in general sums in $\hN_0$
are not coherent with $u$-equivalence,
\emph{i.e.} it can well be the case that $\alpha\ueq\alpha'$
and $\beta\ueq\beta'$, but $\alpha+\beta\nueq\alpha'+\beta'$.

\smallskip
We now extend the notion of generated ultrafilter and also consider
elements
$\nu\in{}^{k*}\N$ in iterated hyper-extensions of $\N$,
by putting
$$\UU_\nu\ =\ \{A\subseteq\N_0\mid \nu\in{}^{k*}A\}.$$

The $u$-equivalence relation is extended accordingly
to all pairs of numbers in the following union
$${}^\star\N_0\ =\ \bigcup_{k\in\N}{}^{k*}\N_0.$$

In general, for every $A\subseteq \N$,
the set ${}^\star A=\bigcup_{k\in\N}{}^{k*}A$
can be seen as the direct limit of
the finitely iterated hyper-extensions of $A$;
and similarly for functions. In consequence,
the map $\star$ itself is a nonstandard embedding,
\emph{i.e.} it satisfies the \emph{transfer principle}.
Since $*$ is assumed to be $\mathfrak{c}^+$-enlarging,
it cab be easily verified that the same
property also holds for $\star$.

\smallskip
Remark that the above definitions are coherent.
In fact, by starting from the equivalence
$n\in A\Leftrightarrow n\in\hA$ which holds
for all $n\in\N_0$, one can easily show
that $\nu\in{}^{k*}A\Leftrightarrow\nu\in{}^{h*}A$
for all $\nu\in{}^{k*}\N_0$ and $h>k$.
In consequence, for all $\nu\in{}^\star\N_0$,
one has $\nu\ueq{}^*\nu$, and hence
${}^{k*}\nu\ueq{}^{h*}\nu$ for all $k,h$.
(A detailed study of $\ueq$-equivalence in ${}^\star\N_0$
can be found in \cite{lu-tesi}.)

\smallskip
We shall use the following characterization of pseudo-sums
of ultrafilters.

\medskip
\begin{proposition}\label{prop-ultrasum}
Let $\alpha,\beta\in\hN_0$ and $A\subseteq\N_0$. Then
$A\in\UU_\alpha\oplus\UU_\beta$ if and only if the sum
$\alpha+{}^*\beta\in{}^{**}A$.
\end{proposition}

\begin{proof}
Consider the set $\widehat{A}=\{n\in\N_0\mid A-n\in\UU_\beta\}$,
and notice that its hyper-extension
${}^*\widehat{A}={}^*\{n\in\N_0\mid n+\beta\in\hA\}=
\{\gamma\in\hN_0\mid\gamma+{}^*\beta\in{}^{**}A\}$.
Then the following equivalences yield the thesis:
$$A\in\UU_\alpha\oplus\UU_\beta\ \Longleftrightarrow\
\widehat{A}\in\UU_\alpha\ \Longleftrightarrow\
\alpha\in{}^*\widehat{A}\ \Longleftrightarrow\
\alpha+{}^*\beta\in{}^{**}A.$$
\end{proof}

\medskip
We already mentioned that when $\alpha\ueq\alpha'$
and $\beta\ueq\beta'$, one cannot conclude that
$\alpha+\beta\nueq\alpha'+\beta'$.
However, under the same assumptions, one has that
$\alpha+{}^*\beta\ueq\alpha+{}^*\beta'$ in $\hhN$,
as they generate the same ultrafilter
$\UU_\alpha\oplus\UU_\beta=\UU_{\alpha'}\oplus\UU_{\beta'}$.

\smallskip
The characterization of pseudo-sums as given
above can be extended to linear combinations of
ultrafilters in a straightforward manner.

\medskip
\begin{corollary}\label{linearcomb}
For every $\xi_0,\ldots,\xi_k\in\hN_0$, and for every
$a_0,\ldots,a_k\in\N_0$, the linear combination
$a_0\U\oplus \ldots\oplus a_k\U$ is the ultrafilter
generated by the element
$a_0\xi+\ldots+a_k{}^{k*}\xi\in{}^{(k+1)*}\N_0$.
\end{corollary}

\medskip
The class of idempotent ultrafilters was first isolated
to provide a simplified proof of
Hindman's Theorem, a cornerstone of combinatorics
of numbers.

\medskip
\noindent
\textbf{Theorem} (Hindman)
\emph{For every finite coloring $\N=C_1\sqcup\ldots\sqcup C_r$
there exists an infinite set $X=\{x_1<x_2<\ldots<x_n<\ldots\}$
such that all its finite sums are monochromatic,
\emph{i.e.} there exists $i$ such that:}
$$\text{FS}(X)\ =\ \left\{\sum_{i\in F}x_i\,\Big|\,
F\subset \N\ \text{nonempty finite}\right\}\subseteq C_i.$$

\medskip
Starting from the ultrafilter proof of the above theorem,
a whole body of new combinatorial results have been then
obtained by exploiting the algebraic properties of the space
$(\beta\N_0,\oplus)$ and of its generalizations
(see the monograph \cite{hs}).

\smallskip
By Proposition \ref{prop-ultrasum}, it
directly follows that

\medskip
\begin{proposition}\label{prop-idempotent}
Let $\xi\in\hN$.
The ultrafilter $\UU_\xi$ is idempotent if
and only if $\xi\ueq\xi+{}^*\xi$.
\end{proposition}

\medskip
Next, we show a general result connecting linear combinations of a
given idempotent ultrafilter, and $u$-equivalence of the
corresponding strings of coefficients.

\medskip
\begin{theorem}\label{th-idemstring}
Let $a_0,a_1,\ldots,a_k,b_0,b_1,\ldots,b_h\in\N_0$.
Then the following are equivalent:

\smallskip
\begin{enumerate}
\item
$\langle a_0,a_1,\ldots,a_k\rangle\,\ueqs\,\langle b_0,b_1,\ldots,b_h\rangle$.

\smallskip
\item
For every idempotent ultrafilter $\U$:
$$a_0\,\U\oplus a_1\,\U\oplus \ldots \oplus a_k\,\U\,=\,
b_0\,\U\oplus b_1\,\U\oplus \ldots \oplus b_h\,\U.$$

\smallskip
\item
For every $\xi\in\hN_0$ such that the generated ultrafilter $\UU_\xi$
is idempotent:
$$a_0\,\xi+a_1{}^*\xi+\ldots+a_k{}^{k*}\xi\ \ueq\
b_0\,\xi+b_1{}^*\xi+\ldots+b_h{}^{h*}\xi.$$
\end{enumerate}
\end{theorem}

\begin{proof}
$(1)\Rightarrow(2)$.
For every string $\sigma=\langle d_0,d_1,\ldots,d_n\rangle$
of numbers $d_i\in\N_0$, and for every idempotent ultrafilter $\U$,
denote by
$$\oplus_\U(\sigma)\ =\
d_0\,\U\oplus d_1\,\U\oplus \ldots \oplus d_n\,\U.$$
We have to show that
$\sigma\ueqs\tau\Rightarrow\oplus_\U(\sigma)=\oplus_\U(\tau)$.
By agreeing that
$\oplus_\U(\varepsilon)=\UU_0$, one trivially has
$\oplus_\U(\varepsilon)=\oplus_\U(\langle 0\rangle)$.
Moreover, $\oplus(\langle a\rangle)=\oplus_\U(\langle a,a\rangle)$
because $a\,\U\oplus a\,\U=a(\U\oplus\U)=a\,\U$.
Now let $\sigma\ueqs\sigma'$ and $\tau\ueqs\tau'$,
where we assume by inductive hypothesis that
$\oplus_\U(\sigma)=\oplus_\U(\sigma')$ and
$\oplus_\U(\tau)=\oplus_\U(\tau')$.
Then, by associativity of the pseudo-sum, it follows that
$$\oplus_\U(\sigma^\frown\tau)\ =\ [\oplus_\U(\sigma)]\oplus[\oplus_\U(\tau)]\ =\
[\oplus_\U(\sigma')]\oplus[\oplus_\U(\tau')]\ =\ \oplus_\U(\sigma'^\frown\tau').$$

\smallskip
$(2)\Rightarrow(1)$.
Assume that
$\langle a_0,a_1,\ldots,a_k\rangle\not\!\!\!\ueqs\langle b_0,b_1,\ldots,b_h\rangle$.
By the previous implication, we can assume without loss
of generality that
$a_i\ne a_{i+1}$ for $i<k$, and that and $b_j\ne b_{j+1}$ for $j<h$.
Then we apply the following known result:

\smallskip
\begin{itemize}
\item
(\cite{ma} Theorem 2.19.)
\emph{Let $a_0,\ldots,a_k,b_0,\ldots,b_h\in\N$ so
that $a_i\ne a_{i+1}$ and $b_j\ne b_{j+1}$ for any $i<k$ and $j<h$.
If $a_0\,\U\oplus\ldots\oplus a_k\,\U=b_0\,\U\oplus\ldots\oplus b_h\,\U$
for some idempotent $\U$ then $k=h$ and $a_i=b_i$ for all $i$.}
\end{itemize}

\smallskip
$(2)\Leftrightarrow(3)$.
By the $\mathfrak{c}^+$-enlargement property, every
ultrafilter $\U$ is generated by some element $\xi\in\hN$,
\emph{i.e.} $\U=\UU_\xi$. So, the thesis
is a particular case of Corollary \ref{linearcomb}.
\end{proof}

\medskip
We shall use the above characterization to justify a neat
formalism which is suitable to handle
idempotent ultrafilters and their linear combinations.
As a first relevant example, let us give a nonstandard ultrafilter
proof of Milliken-Taylor's Theorem, a strengthening of Hindman's Theorem.

\medskip
\begin{lemma}\label{lemma2}
Let $\U$ be an idempotent ultrafilter, and let
$a_0,\ldots,a_k\in\N$. For every
$A\in a_0\U\oplus\ldots\oplus a_k\U$
there exists an infinite set of natural numbers
$$X\ =\ \{x_1<x_2<\ldots<x_n<\ldots\ \}$$
with the property that for every increasing sequence $I_0<\ldots<I_k$
of nonempty finite sets of natural numbers
(\emph{i.e.} $\max I_i<\min I_{i+1}$), the sum
$$\sum_{i\in I_0}a_0 x_i\,+\,\ldots\,+\,
\sum_{i\in I_k}a_k x_i\in A.$$
\end{lemma}

\begin{proof}
By the $\mathfrak{c}^+$-enlargement property, we
can pick $\xi\in\hN$ with $\U=\UU_\xi$.
By Corollary \ref{linearcomb},
$$A\in a_0\U\oplus\ldots\oplus a_k\U\
\Longleftrightarrow\
a_0\xi+a_1{}^*\xi+\ldots+a_k{}^{k*}\xi\in{}^{(k+1)*}A,$$
and so we have that
\begin{eqnarray}
\nonumber
\xi & \in &
\Big\{\eta\in\hN\,\Big|\,
a_0\eta+a_1{}^*\xi+\ldots+a_k{}^{k*}\xi\in{}^{(k+1)*}A\Big\}
\\
\nonumber
{} & = & \hspace{-.15cm}
{}^*\Big\{x\in\N\,\Big|\, a_0 x+a_1\xi+\ldots+a_k{}^{(k-1)*}\xi\in{}^{k*}A\Big\}.
\end{eqnarray}

\smallskip
Since $\langle a_0,a_1,\ldots,a_k\rangle\ueqs\langle a_0,a_0,a_1,\ldots,a_k\rangle$,
we also have that
$a_0\xi+a_0{}^*\xi+a_1{}^{**}\xi+\ldots+a_k{}^{(k+1)*}\xi\in{}^{(k+2)*}A$,
and hence:
\begin{eqnarray}
\nonumber
\xi & \in &
\Big\{\eta\in\hN\,\Big|\,
a_0\eta+a_0{}^*\xi+a_1{}^{**}\xi+\ldots+a_k{}^{(k+1)*}\xi\in{}^{(k+2)*}A\Big\}
\\
\nonumber
{} & = & \hspace{-.15cm}
{}^*\Big\{x\in\N\,\Big|\, a_0 x+a_0\xi+a_1{}^*\xi+\ldots+a_k{}^{k*}\xi\in{}^{(k+1)*}A\Big\}.
\end{eqnarray}

By \emph{transfer}, there exists an element $x_1$ such that

\smallskip
\begin{itemize}
\item
$a_0 x_1+a_1\xi+\ldots+a_k{}^{(k-1)*}\xi\in{}^{k*}A$,\ and

\smallskip
\item
$a_0 x_1+a_0\xi+a_1{}^*\xi+\ldots+a_k{}^{k*}\xi\in{}^{(k+1)*}A$.
\end{itemize}

\smallskip
We now proceed by induction on $n$ and
define elements $x_1<\ldots<x_n$
in such a way that for every increasing sequence
of nonempty finite sets $J_0<\ldots<J_h$ where
$h\le k$ and $\max J_h\le n$,
the following properties are fulfilled:

\smallskip
\begin{enumerate}
\item
$\sum_{s=0}^h\big(\sum_{i\in J_s}a_s x_i\big)\ +a_h\xi + a_{h+1}{}^*\xi+\ldots+
a_k{}^{(k-h)*}\xi\in{}^{(k-h+1)*}A$.

\smallskip
\item
$\sum_{s=0}^h\big(\sum_{i\in J_s}a_s x_i\big)\ +a_{h+1}\xi + a_{h+2}{}^*\xi+\ldots+
a_k{}^{(k-h-1)*}\xi\in{}^{(k-h)*}A$.
\end{enumerate}

\smallskip
Remark that $x_1$ actually satisfies the inductive basis $n=1$,
because in this case one necessarily has $h=0$ and $J_0=\{1\}$.
As for the inductive step, notice that

\smallskip
\begin{itemize}
\item
$\langle a_h, a_{h+1}, \ldots, a_k\rangle\,\ueqs\,
\langle a_h, a_h, a_{h+1}, \ldots, a_k\rangle$, and

\smallskip
\item
$\langle a_{h+1}, a_{h+2}, \ldots, a_k\rangle\,\ueqs\,
\langle a_{h+1}, a_{h+1}, a_{h+2}, \ldots, a_k\rangle$.
\end{itemize}

\smallskip
So, in consequence of the inductive hypotheses $(1)$ and $(2)$ respectively,
one has that

\smallskip
\begin{enumerate}
\item[$(3)$]
$\sum_{s=0}^h\big(\sum_{i\in J_s}a_s x_i\big)\ +a_h\xi + a_h{}^*\xi +
a_{h+1}{}^{**}\xi+\ldots+
a_k{}^{(k-h+1)*}\xi\in{}^{(k-h+2)*}A$.

\smallskip
\item[$(4)$]
$\sum_{s=0}^h\big(\sum_{i\in J_s}a_s x_i\big)\ +a_{h+1}\xi + a_{h+1}{}^*\xi +
a_{h+2}{}^{**}\xi+\ldots+
a_k{}^{(k-h)*}\xi\in{}^{(k-h+1)*}A$.
\end{enumerate}

\smallskip
Now, properties $(1),(2),(3),(4)$ say that
for every increasing sequence of nonempty finite sets
$J_0<\ldots<J_h$ where $h\le k$ and $\max J_h\le n$,
the hyperinteger $\xi\in{}^*\Gamma(J_0<\ldots<J_h)$ where:
\begin{eqnarray}
\nonumber
\Gamma(J_0<\ldots<J_h) & = &
\\
\nonumber
{} & \hspace{-2cm} & \hspace{-3.2cm}
\Big\{m\in\N\,\Big|\,\sum_{s=0}^h\Big(\sum_{i\in J_s}a_s x_i\Big)
+a_h m + a_{h+1}\xi+\ldots+ a_k{}^{(k-h-1)*}\xi\in{}^{(k-h)*}A
\\
\nonumber
{} & \hspace{-2cm} & \hspace{-3cm} \&\ \
\sum_{s=0}^h\Big(\sum_{i\in J_s}a_s x_i\Big)
+a_{h+1}m + a_{h+2}\xi+\ldots+
a_k{}^{(k-h-2)*}\xi\in{}^{(k-h-1)*}A
\\
\nonumber
{} & \hspace{-2cm} & \hspace{-3cm} \&\ \
\sum_{s=0}^h\Big(\sum_{i\in J_s}a_s x_i\Big)
+a_{h}m + a_{h}\xi+a_{h+1}{}^*\xi\ldots+
a_k{}^{(k-h)*}\xi\in{}^{(k-h+1)*}A
\\
\nonumber
{} & \hspace{-2cm} & \hspace{-3cm} \&\ \
\sum_{s=0}^h\Big(\sum_{i\in J_s}a_s x_i\Big)
+a_{h+1}m + a_{h+1}\xi+a_{h+2}{}^*\xi\ldots+
a_k{}^{(k-h-1)*}\xi\in{}^{(k-h)*}A\Big\}
\end{eqnarray}

\smallskip
Then $\xi\in{}^*\Gamma$, where $\Gamma$ is the following finite intersection:
$$\Gamma\ =\ \bigcap_{\substack{J_0<\ldots<J_h \\ h\le k,\ \max J_h\le n}}\Gamma(J_0<\ldots<J_h).$$
The set $\Gamma$ is infinite because its hyper-extension contains
an infinite hyper-natural number, namely $\xi$;
in particular, we can pick an element
$x_{n+1}>x_n$ in $\Gamma$. It now only takes a straightforward
verification to check that $x_1<\ldots<x_n<x_{n+1}$ satisfy
the desired properties, namely $(1)$ and $(2)$
for every sequence of nonempty finite sets
$J_0<\ldots<J_l$ where $l\le k$ and $\max J_l\le n+1$.
\end{proof}

\medskip
As a straight corollary, we obtain

\medskip
\begin{theorem}[Milliken-Taylor]
Let a finite coloring $\N=C_1\sqcup\ldots\sqcup C_r$ be
given. For every choice of $a_0,\ldots,a_k\in\N$
there exists an infinite set
$$X\ =\ \{x_1<x_2<\ldots<x_n<x_{n+1}<\ldots\ \}$$
such that the following sums are monochromatic
for every increasing sequence
$I_0<\ldots<I_k$ of nonempty finite sets:
$$\sum_{i\in I_0}a_0 x_i\ +\ \ldots\ +\
\sum_{i\in I_k}a_k x_i.$$
\end{theorem}

\begin{proof}
Pick any idempotent ultrafilter $\U$, and
consider the linear combination
$\W=a_0\U\oplus\ldots\oplus a_k\U$.
Then take $i$ such that $C_i\in\W$, and apply
the previous Lemma.
\end{proof}

\bigskip
\section{Partition regularity and Rado's Theorem}\label{rado}

We now aim at showing how the introduced nonstandard
approach can be used in partition regularity of linear equations.
Let us start with an example. Recall the following known fact.

\medskip
\begin{theorem}[\cite{bh} Th. 2.10]\label{th-linearcombidem}
Let $\U$ be an idempotent ultrafilter.
Then every set $A\in 2\,\U\oplus\U$ contains
a 3-term arithmetic progression.
\end{theorem}

\medskip
A nonstandard proof of the above theorem is obtained by
the following simple observation.
If the ultrafilter $\U=\UU_\xi$ is idempotent,
then the following three elements of the hyper-hyper-hyper-natural numbers
${}^{***}\N$
are arranged in arithmetic progression, and they all generate
the same ultrafilter $\W=2\,\U\oplus\U$:

\smallskip
\begin{itemize}
\item
$\nu=2\xi+\,0\, +{}^{**}\xi$

\smallskip
\item
$\mu=2\xi+{}^*\xi+{}^{**}\xi$

\smallskip
\item
$\lambda=2\xi+2{}^*\xi+{}^{**}\xi$
\end{itemize}

\medskip
The property that $\nu\ueq\mu\ueq\lambda\ueq 2\xi+{}^*\xi$
directly follows from Theorem \ref{th-idemstring}, since
$\langle 2,0,1\rangle\ueqs\langle 2,1,1\rangle\ueqs
\langle 2,2,1\rangle\ueqs\langle 2,1\rangle$.
Moreover, by Proposition \ref{prop-ultrasum},
the generated ultrafilter is the following:
$$\W\ =\ \U_{2\xi+{}^*\xi}\ =\ \UU_{2\xi}\oplus\UU_\xi\ =\ 2\,\U\oplus\U.$$
If $A\in\W$ then $\nu,\mu,\lambda\in{}^{***}A$, and
the existence of a 3-term arithmetic
progression in $A$ is proved by applying backward \emph{transfer} to
the following property, which holds in ${}^{***}\N$:
$$\exists\,x,y,z\in{}^{***}A\ \text{s.t.}\ \ y-x=z-y>0.$$

We now elaborate on this example to prove a general fact
which connects partition regularity of equations
with $u$-equivalence in the direct limit
$${}^\star\N\ =\ \bigcup_{k\in\N}{}^{k*}\N.$$

Recall the following

\medskip
\begin{definition}
An equation $F(X_1,\ldots,X_n)=0$ is
[injectively] \emph{partition regular} on $\N_0$
if for every finite coloring of $\N_0=C_1\sqcup\ldots\sqcup C_r$
there exist [distinct] monochromatic elements $x_1,\ldots,x_n$
which are a solution, \emph{i.e.} $F(x_1,\ldots,x_n)=0$
and $x_1,\ldots,x_n\in C_i$ for a suitable color $C_i$.
\end{definition}

\medskip
It is a well-known fact that partition regularity is intimately
connected with ultrafilters. In particular,
recall the following:

\medskip
\begin{itemize}
\item
\emph{$F(X_1,\ldots,X_n)=0$ is [injectively]
partition regular on $\N_0$ if and only if
there exists an ultrafilter $\V$ on $\N_0$
such that in every $A\in\V$ one finds [distinct]
elements $x_1,\ldots,x_n\in A$ with $F(x_1,\ldots,x_n)=0$.\footnote
{~A proof of this equivalence can be found \emph{e.g.} in \cite{hs} \S 3.1.}}
\end{itemize}

\smallskip
When the above property is satisfied, we
say that the ultrafilter $\V$ is a \emph{witness}
of the [injective] partition regularity of $F(X_1,\ldots,X_n)=0$.

\smallskip
A useful nonstandard characterization holds.

\medskip
\begin{theorem}\label{partitionregularity}
Let the nonstandard embedding $*$ satisfy the
$\mathfrak{c}^+$-enlargement property.
Then an ultrafilter $\V$ on $\N_0$ witnesses
the [injective] partition regularity of the
equation $F(X_1,\ldots,X_n)=0$ if and only if there exists
[distinct] hyper-natural numbers $\xi_1,\ldots,\xi_n\in\hN_0$
such that $\V=\UU_{\xi_1}=\ldots=\UU_{\xi_n}$
and ${}^*F(\xi_1,\ldots,\xi_n)=0$.
\end{theorem}

\begin{proof}
Assume first that the ultrafilter $\V$ is a witness.
For $A\in\V$, let
$$\Gamma(A)\ =\ \{(x_1,\ldots,x_n)\in A^n\mid [x_i\ne x_j \text{ for }i\ne j] \ \&\
F(x_1,\ldots,x_n)=0\}.$$

Since $\Gamma(A)\cap\Gamma(B)=\Gamma(A\cap B)$, by the hypothesis
it follows that the family
$\{\Gamma(A)\mid A\in\V\}$ satisfies the finite
intersection property and hence, by $\mathfrak{c}^+$-enlargement,
we can pick $(\xi_1,\ldots,\xi_n)\in\bigcap_{A\in\V}{}^*\Gamma(A)$.
Then it is readily checked
from the definitions that the [distinct] components
$\xi_1,\ldots,\xi_n$ are such that $\UU_{\xi_1}=\ldots=\UU_{\xi_n}=\V$
and ${}^*F(\xi_1,\ldots,\xi_n)=0$.

\smallskip
Conversely, let $A\in\V=\UU_{\xi_1}=\ldots=\UU_{\xi_n}$.
By applying backward \emph{transfer} to the property:
``There exist [distinct] $\xi_1,\ldots,\xi_n\in\hA$ such that
${}^*F(\xi_1,\ldots,\xi_n)=0$", one obtains the existence
of [distinct] $x_1,\ldots,x_n\in A$
such that $F(x_1,\ldots,x_n)=0$, as desired.
\end{proof}

\medskip
We can finally prove the result that was used
in Section \ref{sec-uequivalence} to prove
an ultrafilter version of Rado's Theorem, namely
Theorem \ref{th-rado}.

\medskip
\begin{theorem}\label{linear-pr}
Let $a_0,\ldots,a_n\in\N_0$, and assume that
there exist [distinct] polynomials $P_i(X)$ such that
$$P_1(X)\,\ueqs\, \ldots\,\ueqs\, P_k(X)\,\ueqs\,
\sum_{i=0}^n a_iX^i\ \ \text{and}\ \
c_1 P_1(X)+\ldots+c_k P_k(X)=0.$$
Then for every idempotent ultrafilter $\U$ and for every
$A\in a_0\U\oplus\ldots\oplus a_n\U$,
there exist [distinct]
$x_i\in A$ such that $c_1 x_1+\ldots+c_k x_k=0$.
\end{theorem}

\begin{proof}
For $i=1,\ldots,k$, let the polynomial $P_i(X)=\sum_{j=0}^{n_i}b_{ij}X^j$
correspond to the string of coefficients $\langle b_{i0},b_{i1},\ldots,b_{in_i}\rangle$.
Given the idempotent ultrafilter $\U$, pick an hyper-natural
number $\xi\in\hN$ such that $\UU_\xi=\U$ (this is always possible
by $\mathfrak{c}^+$-enlargement), and consider the numbers
$$\zeta_i\ =\ b_{i0}\xi+b_{i1}{}^*\xi+\ldots+b_{in_i}{}^{n_i*}\xi\,\in\,
{}^{(n_i+1)*}\N\,\subset\,{}^\star\N.$$

Since $\xi\hN$ is infinite, for every $d\in\N_0$ one has
$d<\xi$, $d\,\xi<{}^*\xi$, $d\,{}^*\xi<{}^{**}\xi$, and so forth.
In consequence, by the hypothesis $c_1 P_1(X)+\ldots+c_k P_k(X)=0$,
it directly follows that
$c_1 \zeta_1+\ldots+c_k \zeta_k=0$.
Moreover, by the hypotheses $P_i(X)\ueqs\sum_{i=0}^n a_iX^i$,
Theorem \ref{th-idemstring} guarantees that
$$\UU_{\zeta_1}\ =\ \ldots\ =\ \UU_{\zeta_n}\ =\
a_0\U\oplus\ldots\oplus a_n\U.$$

The thesis is finally reached applying the previous
Theorem \ref{partitionregularity}.
(Recall that the nonstandard embedding
$\star$ is $\mathfrak{c}^+$-enlarging,
because the starting nonstandard embedding $*$ was.)
\end{proof}

\bigskip
\section{Final remarks}\label{fr}

A similar characterization of partition regularity
as the one given in Theorem \ref{partitionregularity}
can also be proved for (possibly infinite)
systems of equations. It seems worth investigating the use
of such nonstandard characterizations especially for the
study of homogeneous non-linear equations
(along the lines of \cite{lu-tesi})
and of infinite systems, which are research areas
where very little is known
(see \cite{hls1,hls2,bhl}).
In particular, also the notion of \emph{image partition regularity}
would deserve attention.
Another possible direction for further research
is to consider possible extensions of Theorem \ref{th-rado}
which are closer to the most general form of Rado's theorem for systems.

\end{document}